\documentclass[11pt]{amsart}
\usepackage{amsmath,amsthm, amssymb, amsfonts, amscd}
\usepackage{verbatim}

\usepackage[all]{xy}
\usepackage{amsxtra}
\usepackage{url}
\usepackage{enumerate}
\usepackage{color}

\DeclareFontEncoding{OT2}{}{} 
\DeclareTextFontCommand{\textcyr}{\fontencoding{OT2}
    \fontfamily{wncyr}\fontseries{m}\fontshape{n}\selectfont}

\usepackage{comment}

\newtheorem{lemma}{Lemma}[section]
\newtheorem{proposition}[lemma]{Proposition}
\newtheorem{theorem}[lemma]{Theorem}
\newtheorem{corollary}[lemma]{Corollary}

\theoremstyle{definition}

\newtheorem{example}[lemma]{Example}

\newtheorem{defs}[lemma]{Definitions}

\newtheorem{question}[lemma]{Question}

\newtheorem{subsec}[lemma]{}

\theoremstyle{remark}

\newtheorem{remark}[lemma]{Remark}

\numberwithin{equation}{section}

\newcommand{\id}{\mathrm{id}}
\newcommand{\diag}{\mathrm{diag}}

\newcommand{\Pic}{\operatorname{Pic}}

\newcommand{\Lie}{\operatorname{Lie}}

\newcommand{\Inn}{\operatorname{Inn}}

\newcommand{\Spec}{\operatorname{Spec}}
\newcommand{\Proj}{\operatorname{Proj}}

\newcommand{\Gal}{\operatorname{Gal}}
\newcommand{\SL}{\operatorname{SL}}
\newcommand{\PGL}{\operatorname{PGL}}
\newcommand{\GL}{\operatorname{GL}}
\newcommand{\gSL}{\operatorname{\mathbf{SL}}}

\newcommand{\gPGU}{\operatorname{\mathbf{PGU}}}

\newcommand{\Hom}{\operatorname{Hom}}

\newcommand{\Aut}{\operatorname{Aut}}

\newcommand{\Sym}{{\mathfrak{S}}}

\newcommand{\rank}{\operatorname{rank}}
\newcommand{\A}{\mathbb{A}}
\renewcommand{\P}{\mathbb{P}}
\newcommand{\Z}{\mathbb{Z}}
\newcommand{\bbZ}{\mathbb{Z}}

\newcommand{\R}{\mathbb{R}}
\newcommand{\C}{\mathbb{C}}

\newcommand{\Kbar}{\overline{K}}

\def\Out{{\mathrm{Out}}}

\def\ttt{{\mathfrak{t}}}

\def\G{{\mathbb{G}}}

\def\AA{{\mathbf{A}}}
\def\BB{{\mathbf{B}}}
\def\CC{{\mathbf{C}}}

\def\GG{{\mathbf{G}}}
\def\ve{{\varepsilon}}

\def\A{{\mathbb{A}}}

\def\into{\hookrightarrow}

\DeclareFontEncoding{OT2}{}{} 
\DeclareTextFontCommand{\textcyr}{\fontencoding{OT2}
    \fontfamily{wncyr}\fontseries{m}\fontshape{n}\selectfont}

\def\C{{\mathbb{C}}}

\def\Z{{\mathbb{Z}}}

\def\Hom{{\rm Hom}}

\newcommand{\isoto}{\overset{\sim}{\to}}
\def\SAut{{\mathrm{SAut}}}

\def\SAut{{\rm SAut}}

\def\twisted{{(\Sym_3,\Gamma)\text{\rm -twisted}}}
\def\SS{{\text{\rm natural}}}

\def\sign{{\rm sign\,}}
\def\sS{{{\rm SU}_3}}

\def\SU{{\bf SU}}
\def\SL{{\bf SL}}
\def\GL{{\bf GL}}
\newcommand{\lt}{\mathfrak{t}}
\newcommand{\Lt}{\mathfrak{t}_L}
\def\bbP{{\mathbb{P}}}
\newcommand{\birat}{\overset{\simeq}{\dashrightarrow}}
\def\one{{\mathbf{1}_3}}

\newcommand{\calD}{\mathcal{D}}
\newcommand{\calO}{\mathcal{O}}

\newcommand{\calC}{\mathcal{C}}

\newcommand{\frakS}{\mathfrak{S}}

\newcommand{\fraka}{\mathfrak{a}}

\renewcommand{\bbP}{\mathbb{P}}

\newcommand{\bbR}{\mathbb{R}}

\renewcommand{\bbZ}{\mathbb{Z}}
\newcommand{\bbC}{\mathbb{C}}
\newcommand{\bbT}{\mathbb{T}}
\newcommand{\bbF}{\mathbb{F}}

\newcommand{\bbS}{\mathbb{S}}

\newcommand{\bfF}{\mathbf{F}}

\renewcommand{\Gal}{\textup{Gal}}

\renewcommand{\GL}{{\bf GL}}
\renewcommand{\Spec}{\textup{Spec}}
\renewcommand{\Proj}{\textup{Proj}}

\renewcommand{\Pic}{\textup{Pic}}

\renewcommand{\PGL}{{\bf PGL}}
\newcommand{\PSL}{{\bf PGL}}

\newcommand{\Sp}{{\bf Sp}}

\renewcommand{\Hom}{\textup{Hom}}
\renewcommand{\id}{\textup{id}}

\newcommand{\da}{\dasharrow}

\renewcommand{\Aut}{\textup{Aut}}
\renewcommand{\rank}{\textup{rank}}

\renewcommand{\SL}{{\bf SL}}
\newcommand{\beq}{\begin{equation}}
\newcommand{\eeq}{\end{equation}}

\newcommand{\PSU}{{\bf PGU}}
\newcommand{\SO}{{\bf SO}}
\renewcommand{\SU}{{\bf SU}}

\def\PSU{{\bf PGU}}
\def\PSL{{\bf PGL}}
\def\charact{{\rm char}}
\def\SS{{\mathfrak{S}}}

\def\der{{\rm der}}
\def\tor{{\rm tor}}
\def\onto{{\twoheadrightarrow}}

\def\BB{{\bf B}}
\def\PGU{{\bf PGU}}
\def\tr{{\rm tr\,}}
\def\U{{\bf U}}

\def\G{{\mathbb{G}}}
\def\dasheq{{\stackrel{\simeq}{\dasharrow}}}

\def\Kbar{{\overline{K}}}

\def\sS{{\SU_3}}
\def\St{{\rm St}}
\def\Tw{{\rm Tw}}

\newcommand{\PP}{{\mathbb{P}}}

\title[Real Cayley groups]
{
Real reductive Cayley groups of rank 1 and 2\\
}

\author{Mikhail Borovoi\\ with an appendix by Igor Dolgachev}
\address{Borovoi: Raymond and Beverly Sackler School of Mathematical Sciences,
Tel Aviv University, 69978 Tel Aviv, Israel}
\email{borovoi@post.tau.ac.il}
\thanks{Borovoi was partially supported
by the Hermann Minkowski Center for Geometry}
\address{Dolgachev: Department of Mathematics, University of Michigan,, 525 E. University
Av., Ann Arbor, MI, 49109, USA}
\email{idolga@umich.edu}

\subjclass[2010]{20G20, 20G15, 14E05}

\keywords{Linear algebraic group, Cayley group, Cayley map,
equivariant birational isomorphism, algebraic surface, elementary link}

\begin{document}


\begin{abstract}
A linear algebraic group $G$ is over a field $K$ is called a Cayley $K$-group if it admits
a Cayley map, i.e., a $G$-equivariant $K$-birational isomorphism
between the group variety $G$ and its Lie algebra.
We classify real reductive algebraic groups of absolute rank 1 and 2 that are Cayley $\R$-groups.
\end{abstract}

\maketitle

\section{Introduction}
Let $G$ be a connected linear algebraic group defined over a field $K$,
and let $\Lie(G)$ denote its Lie algebra.
The following definitions are due to Lemire, Popov and Reichstein \cite{LPR}:

\begin{defs}[\cite{LPR}]
A {\em Cayley map for $G$} is a $K$-birational isomorphism
$G\birat\Lie(G)$ which is $G$-equivariant with respect to
the action of $G$ on itself by conjugation and the action
of $G$ on $\Lie(G)$ via the adjoint representation.
A linear algebraic $K$-group is called a {\em Cayley group} if it admits a Cayley map.
A linear algebraic $K$-group is called a {\em stably Cayley group} if  $G\times_K (\mathbb{G}_{m,K})^r$ is Cayley for some $r\ge 0$,
where $\mathbb{G}_{m,K}$ denotes the multiplicative group.
\end{defs}

Lemire, Popov and Reichstein \cite{LPR} classified  Cayley and stably Cayley simple groups
over an algebraically closed field $k$ of characteristic 0.
Borovoi, Kunyavski\u\i, Lemire and Reichstein \cite{BKLR} classified  stably Cayley simple $K$-groups,
and later Borovoi and Kunyavski\u\i\ \cite{BK} classified  stably Cayley semisimple $K$-groups,
over an arbitrary field $K$ of characteristic 0.
Clearly any Cayley $K$-group is stably Cayley.
In the opposite direction,
some of the stably Cayley $K$-groups are known to be Cayley, see \cite[Examples 1.9, 1.11 and 1.16]{LPR}.
For other stably Cayley $K$-groups, it is a difficult problem to determine whether they are Cayley or not.
By \cite[Lemma 5.4(c)]{BKLR} the answer to the question whether a $K$-group is Cayley depends only on the
equivalence class of $G$ up to inner twisting.

By \cite[Corollary 7.1]{BKLR} all the reductive $K$-groups of rank $\le 2$ over a field $K$ of characteristic 0
are stably Cayley (by the rank we always mean the {\em absolute} rank).
We would like to know, which of those stably Cayley $K$-groups of rank $\le 2$ are Cayley.

The case of a simple group of type $\GG_2$ was settled in \cite[\S\,9.2]{LPR} and Iskovskikh's papers \cite{Isk2}, \cite{Isk3}.
Namely, a simple group of type $\GG_2$ over an algebraically closed field $k$ of characteristic 0 is not Cayley.
Hence no simple $K$-group of type $\GG_2$ over a field $K$ of characteristic 0 is Cayley.

Popov  \cite{popov-luna} proved in 1975 that,
contrary to what was expected (cf.~\cite[Remarque, p.~14]{L}),
the group $\SL_3$ over an {\em algebraically closed field}  $k$
of characteristic 0 is Cayley; see \cite[Appendix]{LPR}
for  Popov's original proof, and \cite[\S\,9.1]{LPR}
for an alternative proof.

Here we are interested in $\R$-groups, where $\R$ denotes the field of real numbers.
If $G$ is an {inner} form of a  split reductive $\R$-group, and
$G_\C:=G\times_\R \C$ is {\em stably Cayley} over $\C$,
then by \cite[Remark 1.8]{BKLR} $G$ is {\em stably Cayley} over $\R$.
Similarly, since $\SL_{3,\C}$ is {\em Cayley} over $\C$ by Popov's theorem,
one might expect that the split $\R$-group $\SL_{3,\R}$ is {\em Cayley} over $\R$.
However, this turns out to be false, see Theorem  \ref{thm:SL3} of Appendix A.
On the other hand, the outer form  $\SU_3$ of the split group $\SL_{3,\R}$ is Cayley,
see  Theorem \ref{thm:su3-Igor} of Appendix A and Corollary \ref{cor:su3}.

We recall the classification of reductive $K$-groups of rank $\le 2$.
A reductive $K$-group of rank 1 is either a $K$-torus or a simple $K$-group of type $\AA_1$.
A reductive $K$-group of rank 2 is either not semisimple, or semisimple of type $\AA_1\times\AA_1$,
or simple of one of the types $\AA_2$, $\BB_2=\CC_2$, or $\GG_2$.
If a reductive  $K$-group of rank 2 is not semisimple, then either it is a $K$-torus or it is isogenous to the product of a one-dimensional $K$-torus
and a simple $K$-group of type $\AA_1$.

We recall the classification of {\em real} simple groups of type $\AA_2$.
Such an $\R$-group is isomorphic to one of the groups $\SL_{3,\R}$, $\PGL_{3,\R}$,
$\SU_3$, $\PGU_3$, $\SU(2,1)$, or $\PGU(2,1)$.
Here, following the Book of Involutions \cite[\S\,23]{KMRT}, we write $\PGU_n$ rather than ${\bf PSU}_n$ for the corresponding adjoint group.
We write $\SU(2,1)$ and  $\PGU(2,1)$ for the (inner) forms of $\SU_3$ and  $\PGU_3$, respectively,
 corresponding to the Hermitian form with diagonal matrix $\diag(1,1,-1)$.

In this paper we classify real reductive algebraic groups of  rank $\le 2$ that are Cayley.
To be more precise, for each real reductive group of rank 1 or 2 (up to an isomorphism) we determine whether it is Cayley or not:

\begin{theorem}\label{thm:maintheorem}
Let $G$ be a connected reductive algebraic $\R$-group  of absolute rank $\le 2$   over the field $\R$ of real numbers.
If $G$ is simple of type $\GG_2$ or is isomorphic to $\SL_{3,\R}$, or $\PGU_3$, or $\PGU(2,1)$, then $G$ is {\em not} Cayley.
Otherwise $G$ is Cayley.
\end{theorem}

Theorem \ref{thm:maintheorem}  will be proved case by case.
The cases when $G$ is Cayley will be treated by the author in the main text of the paper.
In the case when $G$ is of type $\GG_2$  it is known that $G$ is not Cayley, see above.
The other cases when $G$ is not Cayley (and again the case of $\SU_3$ when $G$ is Cayley)
will be treated by Igor Dolgachev in Appendix A.
\smallskip

Note that by \cite[Corollary 7.1]{BKLR}
any $K$-group $G$ of absolute rank $\le 2$ over a field $K$ of characteristic 0 is stably Cayley, that is,
there exists $r\ge 0$ such that the group $G\times_K \G_{m,K}^r$ is Cayley, where $\G_{m,K}$ denotes the multiplicative group over $K$.
The following theorem, which generalizes \cite[Proposition 9.1]{LPR}, shows that
one can always  take $r=2$.

\begin{theorem}\label{thm:2}
Let $G$ be a connected reductive algebraic $K$-group  of absolute rank $\le 2$ over a field $K$ of characteristic $0$.
If $G$ is of absolute rank $1$, then $G$ is Cayley.
If $G$ is of absolute rank $2$, then $G\times_K \G_{m,K}^2$ is Cayley.
\end{theorem}

The following question generalizes \cite[Remark 9.13]{LPR}.
\begin{question}
Let $G$ be a reductive $K$-group of absolute rank 2 that is not Cayley, for example $\gSL_{3,\R}$.
Is  $G\times_K \G_{m,K}$ a Cayley group?
\end{question}

\begin{question}
Are the $\R$-groups $\PGU_{2n+1}$ Cayley for $n\ge 2$?
(Note that these $\R$-groups are stably Cayley, see \cite[Thm.~1.4]{BKLR}.)
\end{question}

The plan of the rest of the paper is as follows.
In Section \ref{sec:2} we reproduce some examples of Cayley groups from \cite{LPR},
and state some known properties of Cayley groups.
In Section \ref{sec:3} we prove  Theorem \ref{thm:maintheorem} modulo results of Section \ref{sec:4} and of Appendix A.
In Section \ref{sec:4} we treat the  case  $\SU_3$ of Theorem \ref{thm:maintheorem}, using explicit calculations.
In Section \ref{sec:SL3} we prove Theorem \ref{thm:2} (case by case).
In Appendix A, Igor Dolgachev treats
the difficult cases $\SL_{3,\R}$ and  $\PGU_3$ of Theorem \ref{thm:maintheorem} (and again the case $\SU_3$),
using the theory of elementary links due to Iskovskikh \cite{Isk1}, \cite{Isk2}, \cite{Isk3}.
In Appendix B, contributed by the anonymous referee, the case of the group $\PGL_1(A)$
for a central simple algebra $A$  of degree $n$ over a field $K$ of positive characteristic $p$ dividing $n$ is considered.

\medskip
The author is very grateful to Igor Dolgachev for writing Appendix A.
The author thanks the referee for writing Appendix B.
The author thanks Zinovy Reichstein for most helpful comments.

\section{Preliminary remarks}\label{sec:2}

We reproduce some examples from \cite{LPR}. Note that in \cite{LPR} it is always assumed that the characteristic of $K$ is zero,
while we attempt to state these results assuming that $K$ is a field of arbitrary characteristic.

\begin{example}[Cf. {\cite[Ex. 1.9]{LPR}}]\label{ex2.1}
Consider a finite-dimensional associative $K$-algebra $A$ with unit element 1, over a field $K$ of arbitrary characteristic,
and the $K$-group $A^\times$ of invertible elements of $A$. Then  clearly $A^\times$ is Cayley.
In particular, the $K$-group $\GL_{n,K}$ is Cayley.
\end{example}

\begin{example}[Cf. {\cite[Ex. 1.11]{LPR}}]\label{ex2}
Let $A$ be a {\em central simple} $K$-algebra of degree $n$, and assume that $\charact(K)$ does not divide $n$.
For any element $a\in A$ denote by $\tr a$ the trace of the linear operator $L_a$ of left multiplication by $a$ in $A$.
Then $\tr 1=n^2\neq 0\in K$. The argument in \cite{LPR} shows that the quotient group $\PGL_1(A):=A^\times/\G_{m,K}$ is Cayley.

Now assume that $\charact(K)$ divides $n$ and that $4|n$ when $\charact(K)=2$, then again $\PGL_1(A)$ is Cayley, see Theorem \ref{thm:referee}
in Appendix B below.

We see that if  $\charact(K)\neq 2$ or if  $n$ is odd, then the group $\PGL_{n,K}$ is Cayley.
In particular, in arbitrary characteristic the group $\PGL_{3,K}$ is Cayley.
Moreover, if $\charact(K)\neq 2$, then $\PGL_{2,K}$ is Cayley.
On the contrary, if $\charact(K)=2$, then $\PGL_{2,K}$ is not Cayley, see Proposition \ref{prop:referee} in Appendix B below.
\end{example}

\begin{example}[Cf. {\cite[Ex. 1.16]{LPR}}, {\cite[p.~599]{Weil}}]\label{ex3}
Let $A$ be a a finite-dimensional associative $K$-algebra with unit element 1
over a field $K$ of  characteristic $\neq 2$,
and let $\iota$ be an involution (over $K$) of $A$.
Set
$$
G=\{a\in A^\times\mid a^\iota\, a=1\}^0,
$$
where $S^0$ denotes the identity component of an algebraic group $S$.
The Lie algebra of $G$ is the subspace of odd elements of $A$ for $\iota$,
\[ \Lie(G)=\{a\in A\ |\ a^\iota= -a\}. \]
Since $\charact(K)\neq 2$, the formula
\[ a\mapsto (1-a)(1+a)^{-1} \]
defines an equivariant rational map $\lambda\colon G\dashrightarrow \Lie(G)$, and the formula
\[ b\mapsto (1-b)(1+b)^{-1} \]
defines its inverse $\lambda^{-1}\colon \Lie(G)\dashrightarrow G$.
Thus $\lambda$ is a Cayley map and $G$ is a Cayley group over $K$.

We see that if $L/K$ is a separable quadratic extension, then the group $\U_{n,L/K}$ of $n\times n$ unitary matrices in $M_n(L)$ is Cayley over $K$;
that the  group $\Sp_{2n,K}$ is Cayley over $K$, in particular, $\SL_{2,K}\simeq\Sp_{2,K}$ is Cayley;
that the group $\SO(m,n)$ is Cayley over $K$, in particular, the groups $\PGL_{2,K}\simeq\SO(2,1)$
and $\Sp_{4,K}/\mu_{2,K}\simeq\SO(3,2)$ are Cayley, where $\mu_{2,K}=\{\pm 1\}$ is the group of roots of unity of order dividing 2 in $K$.
Here we write $\SO(m,n)$ for the special orthogonal group over $K$ of the diagonal quadratic form $x_1^2+\dots+x_m^2-x_{m+1}^2-\dots-x_{m+n}^2$.
\end{example}

We state some known properties of Cayley groups.

\begin{remark}
\label{rem2.1}
If $G_1$ and $G_2$ are Cayley $K$-groups over an arbitrary field $K$, then evidently $G_1\times_K G_2$ is a Cayley $K$-group.
\end{remark}

\begin{remark}\label{rem2.2}
If $L/K$ is a finite separable field extension and $H$ is a Cayley $L$-group,
then evidently the Weil restriction $R_{L/K} H$ is a Cayley $K$-group.
\end{remark}

\begin{remark}\label{rem2.3}
If $G$ is a Cayley $K$-group over an arbitrary field $K$, and $L/K$ is an arbitrary field extension,
then $G\times_K L$ is evidently  a Cayley $L$-group.
\end{remark}

\begin{proposition}[{\cite[Lemma 5.4(c)]{BKLR}}]\label{rem2.4}
If $G$ is a Cayley $K$-group over an arbitrary field $K$, then all the inner forms of $G$ are Cayley.
In particular, if all the automorphisms of $G$ are inner,
then all the twisted forms of $G$ are inner forms, hence they all are Cayley $K$-groups.
\end{proposition}

The following lemma is a version of \cite[Lemma 5.4(a)]{BKLR} and can be proved similarly.

\begin{lemma}\label{lem5.4(a)}
Let $G$ be a reductive $K$-group and $M$ be a $K$-group, not necessarily connected, acting on $G$, over a field $K$ of characteristic 0.
Consider the induced action of $M$ on $\Lie(G)$.
Let $L/K$ be a Galois extension, and $c\colon \Gal(L/K)\to M(L)$ be a cocycle.
Assume that there exists an $M$-equivariant birational isomorphism $f\colon G\dasheq\Lie(G)$ over $K$.
Then there exists a $_c M$-equivariant birational isomorphism of the twisted varieties
$_c f \colon _c G\dasheq \Lie(_c G)$, where $_c M$ is the twisted group.
\end{lemma}

\begin{proposition}[{\cite[Corollary 6.5]{BKLR}}]\label{prop:cor6.5}
Let $G$ be a reductive $K$-group over a field $K$ of characteristic 0, and let $T\subset G$ be a maximal $K$-torus.
Then $G$ is Cayley if and only if there exists a $W(G,T)$-equivariant birational isomorphism $T\stackrel{\simeq}{\dasharrow} \Lie(T)$
defined over $K$, where the Weyl group $W(G,T)$ is viewed as an algebraic $K$-group.
\end{proposition}
Note that the proof of this (difficult) result uses \cite{CKPR}, where it is assumed that $\charact(K)=0$.

\section{Proof of Theorem \ref{thm:maintheorem}, easy cases}\label{sec:3}

We start proving Theorem \ref{thm:maintheorem} case by case.

\begin{proposition}\label{pr1}
Any connected reductive $K$-group $G$ of (absolute) rank $1$ over a field $K$ of characteristic $\neq 2,3$ is a Cayley group.
\end{proposition}

\begin{proof}
 If $G$ is a torus of rank 1, then $G$ is $K$-rational, see e.g. \cite[\S\, 4.9, Example 6]{Voskresenskii},
hence it is Cayley over $K$.
If $G$ is not a torus, then $G$ is simple of rank 1, hence $G$ is a twisted form of one of the groups $\SL_{2,K},\ \PGL_{2,K}$.
Both these groups are Cayley over $K$, see Example \ref{ex3} and Example \ref{ex2}.
Since all the automorphisms of  $\SL_{2,K}$ and $\PGL_{2,K}$ are inner, by Proposition \ref{rem2.4} $G$ is Cayley.
\end{proof}

\begin{proposition}\label{pr2}
Any connected, reductive and not semisimple $K$-group $G$ of absolute rank $2$
over a field $K$  of characteristic $\neq 2,3$  is a Cayley group.
\end{proposition}

\begin{proof}
If $G$ is a torus of rank 2, then $G$ is $K$-rational,
see  \cite[\S\, 4.9, Example 7]{Voskresenskii}, hence it is Cayley over $K$.
If $G$ is not a torus, denote by $R:=Z(G)^0$ its radical and by $G^\der:=[G,G]$ its commutator subgroup.
Since $G$ is not a torus and not semisimple, $R$ is a one-dimensional torus and $G^\der$ is a simple group of absolute rank 1.
Set $\mu=R\cap G^\der$.
The multiplication in $G$ gives a canonical epimorphism
$\pi\colon R\times_K G^\der\onto G$
with kernel isomorphic to $\mu$.

If this epimorphism is an isomorphism, then $G$ is isomorphic to the product of  two $K$-groups $R$ and $G^\der$ of absolute rank 1.
By Proposition \ref{pr1}, $R$ and $G^\der$ are Cayley over $K$, hence by Remark \ref{rem2.1} $G$ is Cayley.

If the epimorphism  $\pi\colon R\times_K G^\der\onto G$ is not an isomorphism, then $\mu\neq \{1\}$.
It follows that the center $Z(G^\der)\neq 1$, hence the simple group $G^\der$ of absolute rank 1 is not adjoint, hence it is simply connected.
We see that $G^\der _\Kbar\simeq\SL_{2,\Kbar}$, and $\mu_\Kbar=Z(G_\Kbar)=\mu_{2,\Kbar}=\{\pm 1\}$,
where $\Kbar$ is an algebraic closure of $K$.
Thus $G_\Kbar=(\G_{m,\Kbar}\times_\Kbar\, \SL_{2,\Kbar})/\mu_{2,\Kbar}\simeq \GL_{2,\Kbar}$
(here $\mu_{2,\Kbar}$ is embedded diagonally).
This means that $G$ is a $K$-form of $\GL_2$.
By Lemma \ref{lem:GL2} below all the $K$-forms of $\GL_2$ are Cayley and hence, $G$ is Cayley.
\end{proof}

\begin{lemma}\label{lem:GL2}
 Any $K$-form of $\GL_2$ over a field $K$  of characteristic $\neq 2,3$  is a Cayley group.
\end{lemma}

\begin{proof}
Write $\Out(G):=\Aut(G)/\Inn(G)$ for ``the group of outer automorphisms'' of $G$.
Write $G^\tor:=G/G^\der$.
The canonical homomorphism
$$\Aut(G)\to\Aut(G^\der)\times\Aut(G^\tor)$$
gives for $G=\GL_2$ a canonical isomorphism
$$
\Aut(\GL_2)\isoto\Aut(\SL_2)\times\Aut(\G_m).
$$
Since all the elements of $\Aut(\SL_2)$ are {\em inner} automorphisms,
we obtain a canonical isomorphism
$$
\Out(\GL_2)\isoto\Aut(\G_m).
$$
taking the class of an automorphism of $\GL_2$ to the induced automorphism of $(\GL_2)^\tor=\G_m$.
Thus we obtain a bijection of the set of $K$-forms of $\GL_{2,K}$ up to inner twisting
onto the set of $K$-forms of $\G_{m,K}$ up to an isomorphism.
One can easily see that this bijection takes $[\GL_{2,K}]$ to  $[\G_{m,K}]$ and  $[\U_{2,L/K}]$ to $[\U_{1,L/K}]$,
where $L$ runs over the separable quadratic extensions of $K$
and we denote by $[\ ]$ the corresponding equivalence classes.
Since the $K$-groups $\G_{m,K}$ and $\U_{1,L/K}$ are {\em all} the $K$-forms of $\G_m$ up to an isomorphism,
we see that $\GL_{2,K}$ and $\U_{2,L/K}$ are all the $K$-forms of $\GL_2$ up to inner twisting.
Since all these $K$-groups, $\GL_{2,K}$ and $\U_{2,L/K}$, are Cayley, see Examples \ref{ex2.1} and \ref{ex3},
we conclude, using Proposition \ref{rem2.4}, that all the $K$-forms of $\GL_2$ are Cayley.
\end{proof}

\begin{proposition}\label{pr3}
Any connected semisimple  $K$-group $G$ of absolute rank $2$ of type $\AA_1\times \AA_1$
over a field $K$  of characteristic $\neq 2,3$  is a Cayley group.
\end{proposition}

\begin{proof}
In this case the group $G$ decomposes into an almost direct product of two groups of type $\AA_1$  defined either over $K$ or
over a separable quadratic extension $L$ of $K$.
If this almost direct product is direct, then $G$ is either a direct product of two simple  $K$-groups of type $\AA_1$,
and hence is Cayley by Proposition \ref{pr1} and Remark \ref{rem2.1}, or $G$ is of the form $R_{L/K} G'$, where $G'$ is a simple $L$-group of type $\AA_1$,
and we conclude by Proposition \ref{pr1}  that $G'$ is Cayley over $L$, and conclude by Remark \ref{rem2.2} that $G$ is Cayley over $K$.
If this almost direct product is not direct, then $G$ is a twisted form of $\SO_4$,
hence $G$ is an inner form of a special orthogonal group of the form $\SO(K^4, q)$
for some nondegenerate quadratic form $q$ in 4 variables, and $G$ is Cayley by Example \ref{ex3} and Proposition \ref{rem2.4}.
\end{proof}

\begin{proposition}\label{pr4}
Any connected simple  $K$-group $G$ of absolute rank $2$ of type $\BB_2=\CC_2$
over a field $K$  of characteristic $\neq 2,3$ is a Cayley group.
\end{proposition}

\begin{proof}
In this case $G$ is an (inner) twisted form of one of the $K$-groups $\Sp_{4,K}$ and $\Sp_{4,K}/\mu_{2,K}$.
Both these groups are Cayley by Example \ref{ex3}, and using Proposition \ref{rem2.4}, we conclude that $G$ is Cayley.
\end{proof}

\begin{subsec} {\em Proof of Theorem \ref{thm:maintheorem} modulo Theorem \ref{thm:su3} and results of Appendix A.}

The cases when $G$ is not a simple group of type $\GG_2$ or  $\AA_2$ were treated in
 Propositions \ref{pr1}, \ref{pr2}, \ref{pr3}, and  \ref{pr4}.

Any connected simple  $K$-group  of absolute rank $2$ of type $\GG_2$ over a field $K$ of characteristic $0$
is not Cayley, see \cite[\S\,9.2]{LPR} and Iskovskikh's papers \cite{Isk2}, \cite{Isk3} (this was explained in our Introduction).

Let $G$ be a connected simple  $\R$-group of rank 2 of type $\AA_2$. We consider all the possible cases.

The group $\PGL_{3,\R}$ is Cayley by Example \ref{ex2}. The group $\SU_3$ is Cayley
by Corollary \ref{cor:su3}  of Theorem \ref{thm:su3}, see  also  Theorem \ref{thm:su3-Igor} of Appendix A.
Since the group $\SU(2,1)$ is an inner form of $\SU_3$, by  Proposition \ref{rem2.4} it is Cayley as well.

The group $\SL_{3,\R}$ is not Cayley by Theorem \ref{thm:SL3} of Appendix A.
The group $\PGU_3$ is not Cayley by Theorem \ref{thm:PGU3} of Appendix A.
Since the group $\PGU(2,1)$ is an inner form of $\PGU_3$,  by  Proposition \ref{rem2.4} it is not Cayley either.
\qed
\end{subsec}

\section{The group $\SU_3$}\label{sec:4}

\begin{subsec}\label{subsec:W-Gamma}
Let $W$ be a finite group. Let $L/K$ be a finite Galois extension with Galois group $\Gamma=\Gal(L/K)$.
We shall consider $W$-varieties defined over $K$ and $(W,\Gamma)$-varieties defined over $L$.
By a $W$-variety defined over $K$ we mean a $K$-variety $X$ with a $W$-action
$W\to\Aut(X)$.
By a semilinear  action of $\Gamma$ on an $L$-variety $Y$
we mean a homomorphism $\rho\colon \Gamma\to\SAut_{L/K}(Y)$ into the group $\SAut_{L/K}(Y)$
of $L/K$-semilinear automorphisms of $Y$,
such that $\rho(\gamma)$ is a $\gamma$-semilinear automorphism of $Y$ for any $\gamma\in\Gamma$
(see \cite[\S\,1.1]{Borovoi-Duke} and \cite[\S\,1.2]{FSS} for definitions of semilinear automorphisms).
By a $(W,\Gamma)$-variety defined over $L$
we mean an $L$-variety $Y$ with two commuting actions:
an $L$-action of $W$ and a semilinear action of $\Gamma$.
One defines morphisms and rational maps of $(W,\Gamma)$-varieties.
We have a base change functor $X\mapsto X\times_K L$ from the category of $W$-varieties over $K$
to the category of $(W,\Gamma)$-varieties over $L$,
and it is well known that this functor is fully faithful,
i.e., the natural  map
$$
\Hom_{\,W}(X,X')\to \Hom_{(W,\Gamma)}(X\times_K L, X'\times_K L)
$$
is bijective for any two $W$-varieties $X,X'$ defined over $K$.
Similarly, $W$-varieties $X$ and $X'$ over $K$ are $W$-equivariantly birationally isomorphic over $K$
if and only if $X\times_K L$ and $X'\times_K L$ are
$(W,\Gamma)$-equivariantly birationally isomorphic over $L$.
Note that, by Galois descent (see Serre \cite[Ch.~V.20, Cor.~2 of Prop.~12]{Serre2}),
any {\em quasi-projective} $(W,\Gamma)$-variety over $L$
comes from a $W$-variety over $K$; we shall not use this fact, however.
\end{subsec}

\begin{subsec}\label{subsec:K-L}
Let $K$ be a field of characteristic 0.
Assume that $K$ does not contain non-trivial roots of unity of order 3.
Set  $L=K(\zeta)$, where $\zeta^3=1,\ \zeta\neq 1$.
We can also write $L=K(\sqrt{-3})$.
(For example, one can take $K=\R$, $L=\R(\sqrt{-3})=\C$.)
We set $\Gamma=\Gal(L/K)$, $\Gamma=\{\id,\gamma\}$,
and we write the action of $\gamma$ on $a\in L$ as $a\mapsto{}^\gamma a$.

Let $G=\SU(3, L/K,H):=\SU(L^3, H)$, the special unitary group of the $L/K$-Hermitian form with  matrix $H$,
where $H\in M_3(L)$ is a nondegenerate $3\times 3$ Hermitian matrix.
Then $G$ is a simple $K$-group, an outer $L/K$-form of the split $K$-group $\SL_{3,K}$.
Note that  $G=\SU(3,L/K,H)$ is an {\em inner} form of the $K$-group
$\SU_{3,L/K}:=\SU(3,L/K, I_3)$, where $I_3=\diag(1,1,1)$.
\end{subsec}

\begin{theorem} \label{thm:su3}
Let  a field $K$, the quadratic field extension $L=K(\zeta)$ of $K$,
and an Hermitian matrix  $H\in M_3(L)$ be as in \S\,\ref{subsec:K-L}.
Then the  $K$-group $G=\SU(3,L/K,H)$  is  Cayley.
\end{theorem}

Theorem \ref{thm:su3} will be proved below.

\begin{corollary} \label{cor:su3}
The $\R$-groups $\SU_3$ and $\SU(2,1)$ are Cayley.\qed
\end{corollary}

\begin{subsec}
Let $K,L$ be as in \S\,\ref{subsec:K-L}.
Consider the torus $\G_{m,K}^3$
and write the standard   action of $\Sym_3$ on it, given by:
\begin{equation}\label{sigma-action-U}
\sigma(x_1, x_2, x_3):=
(x_{\sigma^{-1}(1)}, x_{\sigma^{-1}(2)}, x_{\sigma^{-1}(3)})
\quad\text{for}\hskip 2mm \sigma\in\Sym_3.
\end{equation}
We consider the $K$-subtorus
\begin{equation*}
T:=\{(x_1,x_2,x_3)\in\G_{m,K}^3\mid x_1x_2x_3=1\}
\end{equation*}
and we set $\lt=\Lie(T)$.

We set   $T_L=T\times_K L$, $\Lt=\Lie(T_L)=\lt\otimes_K L$, then
$$
\Lt=\{(x_1,x_3,x_3)\in L^3\mid x_1+x_2+x_3=0\}.
$$
The group $\Sym_3$ acts on $T_L$ and $\Lt$ by formula \eqref{sigma-action-U}, and $\Gamma$ acts by
$$
{}^\gamma(x_1,x_2,x_3)=({}^\gamma x_1, {}^\gamma x_2, {}^\gamma x_3).
$$

We consider also the $\Gamma$-twisted $(\SS_3,\Gamma)$-varieties $T'_L$ and $\Lt'$:
the same $L$-varieties $T_L$ and $\Lt$ with the same $\SS_3$-actions,
 but with the twisted actions of $\gamma$:
\begin{align*}
&(x_1,x_2,x_3)\mapsto ({}^\gamma x_1^{-1},{}^\gamma x_2^{-1},{}^\gamma x_3^{-1})\quad \text{for } T'_L,\\
&(x_1,x_2,x_3)\mapsto (-{}^\gamma x_1,-{}^\gamma x_2,-{}^\gamma x_3)\quad \text{for }\Lt'.
\end{align*}
These $(\SS_3,\Gamma)$-varieties over $L$ come from some $\SS_3$-varieties $T'$ and $\lt'$ defined over $K$
which are easy to describe, see below.
\end{subsec}

\begin{subsec}
Let $T_\sS$ denote the diagonal maximal torus of $\SU_{3,L/K}$, and let $\lt_\sS$ denote its Lie algebra.
Let $N_\sS$ denote the normalizer of $T_\sS$ in  $\SU_{3,L/K}$, and set $W=N_\sS/T_\sS$.
The finite algebraic group $W$ is canonically isomorphic to the symmetric group $\SS_3$ with trivial Galois action.
We see that $T_\sS$, and  $\lt_\sS$ are $\SS_3$-varieties over $K$.
Furthermore, it is well known that $T_\sS\times_K L$ is canonically isomorphic to $T'_L$ and that
$\lt_\sS\otimes_K L$ is canonically isomorphic to $\Lt'$ as $(\SS_3,\Gamma)$-varieties.
Therefore we set
$$T':=T_\sS,\quad \lt':=\lt_\sS.$$
\end{subsec}

\begin{proposition}\label{prop:birat-tori}
Let $K$ be a field of characteristic 0.
We assume that $K$ contains no nontrivial cube root of 1, and we set $L=K(\zeta)$, where $\zeta^3=1$, $\zeta\neq 1$.
Then the  $(\SS_3,\Gamma)$-varieties $T'_L$ and $\Lt'$ are $(\SS_3,\Gamma)$-equivariantly birationally isomorphic over $L$.
\end{proposition}

\begin{subsec}
{\em Reduction of Theorem \ref{thm:su3} to Proposition \ref{prop:birat-tori}.}
Since our group $\SU(3,L/K,H)$ is an inner form of $\SU_{3,L/K}$,
by Proposition \ref{rem2.4} in order to prove that the group $\SU(3,L/K,H)$ is Cayley,
it suffices to prove that $\SU_{3,L/K}$ is  Cayley.
By Proposition \ref{prop:cor6.5}, the group $\SU_{3,L/K}$ is Cayley if and only if
the $\SS_3$-varieties $T'=T_\sS$ and $\lt'=\lt_\sS$ are $\SS_3$-equivariantly birationally isomorphic over $K$.
The discussion in \S\,\ref{subsec:W-Gamma} shows that they are $\SS_3$-equivariantly birationally isomorphic over $K$
if and only if the $(\SS_3,\Gamma)$-varieties $T'_L$ and $\Lt'$ are  $(\SS_3,\Gamma)$-equivariantly birationally isomorphic over $L$.
Therefore, Theorem \ref{thm:su3}  follows from  the  Proposition \ref{prop:birat-tori}.
\end{subsec}

We give here a proof  of Proposition \ref{prop:birat-tori} which is close to the proof of Proposition 9.1 in \cite{LPR}.
For an alternative proof (in the case $K=\R$) see Appendix A, Theorem \ref{thm:su3-Igor}.

\begin{subsec}\label{3.2}
We consider the variety  $(\G_{m,L}^3/\G_{m,L})_\twisted$, which is just $\G_{m,L}^3/\G_{m,L}$
(with $\G_{m,L}$ imbedded diagonally in $\G_{m,L}^3$) with the following (twisted) $\SS_3$-action and {\em twisted $\Gamma$-action}:
$$
\sigma([x])=[\sigma(x)^{\sign \sigma}],\quad {}^\gamma[x]=[{}^\gamma x^{-1}]\quad\text{for }x\in \G_{m,L}^3,\ \sigma\in \SS_3.
$$
Here we write $[x]\in  \G_{m,L}^3/\G_{m,L}$ for the class of $x\in \G_{m,L}^3$.
We have an $(\SS_3,\Gamma)$-equivariant isomorphism
$$
(\G_{m,L}^3/\G_{m,L})_\twisted\isoto T'_L,\quad [x_1,x_2,x_3]\mapsto (x_2/x_3, x_3/x_1, x_1/x_2).
$$
It remains to prove that $(\G_{m,L}^3/\G_{m,L})_\twisted$ is
$(S_3,\Gamma)$-equivariantly birationally isomorphic to $\Lt'$.
\end{subsec}

\begin{subsec}\label{3.3}
Consider the following (twisted) $\SS_3$-action  and {\em twisted $\Gamma$-action} on the set $\Lt\times \Lt$:
\begin{align*}\label{S3-Gamma-twisted1}
\sigma(x, y):=&
\begin{cases}
 \bigl(\sigma(x), \sigma(y)\bigr)
 &\text{if $\sigma$ is even,} \\
 \bigl(\sigma(y), \sigma(x)\bigr)
 &\text{if $\sigma$ is odd,}
\end{cases}
\hskip 4mm\text{where}\hskip 2mm
\sigma\in\Sym_3,\ x, y\in\Lt,\\
{}^\gamma(x,y):=&({}^\gamma y,{}^\gamma x).
\end{align*}
These actions of $\SS_3$ and $\Gamma$ on $\Lt\times\Lt$
induce actions on the surface $\P(\Lt)\times_L \P(\Lt)$,
on the tensor product $\Lt\otimes_L \Lt$
and on the  3-dimensional projective space $\P(\Lt\otimes_L \Lt)$, and
we write
$$
(\P(\Lt)\times_L \P(\Lt))_\twisted,\
(\Lt\otimes_L \Lt)_\twisted\text{ and }\P(\Lt\otimes_L \Lt)_\twisted
$$
for the corresponding  $(\SS_3,\Gamma)$-varieties.
\end{subsec}

\begin{subsec}\label{3.4}
We claim that the $(\SS_3,\Gamma)$-varieties
$(\G_{m,L}^3/\G_{m,L})_\twisted$ and
$(\bbP(\Lt)\times_L \bbP(\Lt))_\twisted$ are
 $(\SS_3,\Gamma)$-equivariantly birationally isomorphic.
We write $[t]\in\P(\Lt)$ for the class of $t\in\Lt$.
Consider the rational map
\begin{align*}
\varphi:
&(\G_{m,L}^3/\G_{m,L})_\twisted\dasharrow (\bbP(\Lt)\times_L \bbP(\Lt))_\twisted\\
&[x]\mapsto \bigl(\bigl[x-\tau(x)\one\bigr], \bigl[x^{-1}-\tau(x^{-1})\one\bigr]\bigr),
\end{align*}
where $\tau(x_1,x_2,x_3)=(x_1+x_2+x_3)/3$ and $\one=(1,1,1)\in L^3$.
It is immediately seen that $\varphi$ is $(\SS_3,\Gamma)$-equivariant.
An inverse rational map to $\varphi$ was constructed in \cite[Proof of Prop.~9.1, Step 1]{LPR}.
Thus $\varphi$ is an $(\SS_3,\Gamma)$-equivariant birational isomorphism.
\end{subsec}

\begin{subsec}\label{3.5}
Consider the Segre embedding
\begin{equation*}
(\bbP(\Lt) \times_L \bbP(\Lt))_\twisted
\hookrightarrow \bbP(\Lt \otimes_L \Lt)_\twisted
\end{equation*}
given by $([x],[y])\mapsto [x\otimes y]$,
it is $(\SS_3,\Gamma)$-equivariant.
Its image is a quadric $Q$ in
$\bbP(\Lt\otimes_L \Lt)$ described as follows. Choose a
basis $D_1 := {\rm diag}(1, \zeta, \zeta^2)$,
$D_2 := {\rm diag}(1, \zeta^2, \zeta)$ of
$\Lt$, where $\zeta$ is our primitive cube root
of unity. Set $D_{ij} = D_i \otimes D_j$.
Then
\begin{equation}\label{quadric}
Q = \{ (\alpha_{11}: \alpha_{12}:
\alpha_{21}: \alpha_{22}) \mid \alpha_{11}
\alpha_{22} = \alpha_{12} \alpha_{21} \},
\end{equation}
where $ (\alpha_{11}:
\alpha_{12}: \alpha_{21}: \alpha_{22})$ is the
point of $\bbP(\Lt \otimes_L \Lt)$ corresponding
to
$$\alpha_{11} D_{11} + \alpha_{12} D_{12} +
\alpha_{21} D_{21} + \alpha_{22} D_{22}\in\Lt\otimes_L \Lt.
$$
\end{subsec}

\begin{subsec}\label{3.7}
We denote by $V_{11,22}$ the 2-dimensional subspace in $(\Lt\otimes_L \Lt)_\twisted$
with the basis $D_{11}, D_{22}$, and we denote by $V_{12}$ and $V_{21}$ the one-dimensional subspaces
generated by $D_{12}$ and $D_{21}$, respectively.
An easy calculation shows  that the subspace $V_{11,22}$ is $\SS_3$-invariant and $\Gamma$-invariant,
and that the basis vectors $D_{12}$ and $D_{21}$ are $\Sym_3$-fixed and $\Gamma$-fixed.

Consider the stereographic projection
$Q\dashrightarrow \P(V_{11,22}\oplus V_{12})$
from the $(\SS_3,\Gamma)$-fixed $L$-point $x_{21}:=[D_{21}]=(0:0:1:0)\in Q(L)$ to the $(\SS_3,\Gamma)$-invariant plane
$\P(V_{11,22}\oplus V_{12})$.
This stereographic projection is a $(\Sym_3,\Gamma)$-equivariant birational isomorphism.
Furthermore, the embedding
$$
 V_{11,22}\into \P(V_{11,22}\oplus V_{12}),\quad x\mapsto [x+D_{12}]
$$
is an $(\SS_3,\Gamma)$-equivariant birational isomorphism.
Thus the quadric $Q$ is  $(\SS_3,\Gamma)$-equivariantly
birationally isomorphic to the  vector space $V_{11,22}$.
Since the 2-dimensional $(\SS_3,\Gamma)$-vector spaces $V_{11,22}$ and $\Lt$ are isomorphic
(the map of bases $D_{11}\mapsto D_2,\ D_{22}\mapsto D_1$ induces an $(\SS_3,\Gamma)$-isomorphism  $V_{11,22}\isoto\Lt$),
and $\Lt$ is isomorphic to $\Lt'$  (an isomorphism is given by
$(x_i)\mapsto (\sqrt{-3}\cdot x_i)$),
we conclude that $Q$ is  $(\SS_3,\Gamma)$-equivariantly
birationally isomorphic to  $\Lt'$.

Thus $T'_L$ is  $(\SS_3,\Gamma)$-equivariantly
birationally isomorphic to  $\Lt'$.
This completes the proofs of Proposition \ref{prop:birat-tori},  Theorem \ref{thm:su3}, and Corollary \ref{cor:su3}.
\qed

\end{subsec}

\section{The groups  $G\times \G_m^2$}\label{sec:SL3}

In this section we  prove Theorem \ref{thm:2}.
Let $K$ be a field of characteristic 0, and let $\Kbar$ be a fixed algebraic closure of $K$.

Let $\GG_{2,K}$ denote the {\em split} $K$-group of type $\GG_2$.
Since by \cite[Proposition 9.10]{LPR}, the group $\GG_{2,\Kbar}$ is not Cayley over $\Kbar$,
we see that $\GG_{2,K}$ is not Cayley.

\begin{proposition}\label{prop:G2}
For any field $K$ of characteristic 0,
the split $K$-group $\GG_{2,K}\times_K \G_{m,K}^2$ is Cayley.
\end{proposition}

\begin{corollary}\label{cor:G2}
For any $K$-group $G$ of type $\GG_2$ over a field $K$ of characteristic 0,
the $K$-group $G\times_K \G_{m,K}^2$ is Cayley.
\end{corollary}
\begin{proof}
Since $G\times_K \G_{m,K}^2$ is an inner form of $\GG_{2,K}\times_K \G_{m,K}^2$,
by Proposition \ref{rem2.4} the corollary follows from Proposition \ref{prop:G2}.
\end{proof}

\begin{subsec}\label{subsec:G2}
Let $K$ be a field of characteristic 0.
We define a $K$-torus $T$  by
\begin{equation*}
T:=\{(x_1,x_2,x_3)\in\G_{m,K}^3\mid x_1x_2x_3=1\}.
\end{equation*}
We define a $K$-action of $\Sym_3$ on $T$ by
\begin{equation*}
\sigma(x_1, x_2, x_3):=
(x_{\sigma^{-1}(1)}, x_{\sigma^{-1}(2)}, x_{\sigma^{-1}(3)})
\quad\text{for}\hskip 2mm \sigma\in\Sym_3.
\end{equation*}
We define a $K$-action of $\SS_2$ on $T$ by
$$
\ve(t)=t^{-1} \text{ for } t\in T,
$$
where $\ve$ is the nontrivial element of $\SS_2$.
We obtain a $K$-action of $\SS_3\times \SS_2$ on $T$.
Set $\ttt=\Lie(T)$, then $\SS_3\times \SS_2$ acts on $\ttt$.
We may regard $T$ as a split maximal torus of  $\GG_{2,K}$, and $\SS_3\times \SS_2$ as the corresponding Weyl group,
then $T\times_K \G_{m,K}^2$ is a maximal torus of $\GG_{2,K}\times_K  \G_{m,K}^2$.
\end{subsec}

\begin{proposition}[\cite{LPR}]\label{prop:LPR-G2}
\label{prop:G2-torus}
For an arbitrary field $K$ of characteristic 0, the $K$-varieties $T\times_K \G_{m,K}^2$ and $\ttt\times_K \A_K^2$
are $\SS_3\times\SS_2$-equivariantly birationally isomorphic over $K$.
\end{proposition}

\begin{proof}
This is proved in \cite{LPR} in the proof of Proposition 9.11.
The authors assume that $K$ is an algebraically closed field of characteristic 0,
but the proof goes through for any field $K$  of characteristic $\neq 2,3$.
\end{proof}

\begin{proof}[Proof of Proposition \ref{prop:G2}]
By Proposition \ref{prop:cor6.5}, our proposition follows from Proposition \ref{prop:G2-torus}.
\end{proof}

\begin{corollary}\label{cor:SL3-torus}
The $K$-varieties $T\times_K \G_{m,K}^2$ and $\ttt\times_K \A_K^2$ of Proposition \ref{prop:G2-torus}
are $\SS_3$-equivariantly birationally isomorphic over $K$ (with respect to the standard embedding $\Sym_3\into\Sym_3\times\Sym_2$).
\end{corollary}

\begin{proof}
The $\SS_3\times\SS_2$-equivariant birational isomorphism  of Proposition \ref{prop:G2-torus} is, in particular,
$\SS_3$-equivariant.
\end{proof}

\begin{proposition}\label{prop:SL3}
For any field $K$ of characteristic 0, the $K$-group $\gSL_{3,K}\times_K \G_{m,K}^2$ is Cayley.
\end{proposition}

\begin{proof}
We regard $T$ as a split maximal torus of $\gSL_{3,K}$ and $\SS_3$ as the corresponding Weyl group,
then  $T\times_K\, \G_{m,K}^2$ is a maximal torus of  $\gSL_{3,K}\times_K\, \G_{m,K}^2$.
Now by Proposition \ref{prop:cor6.5}, our proposition follows from Corollary \ref{cor:SL3-torus}.
\end{proof}

\begin{subsec}
Let $T$ be the $\SS_3\times\SS_2$-torus over $K$ of Section \ref{subsec:G2}.
Let $L/K$ be an arbitrary quadratic extension. Write $\Gamma=\Gal(L/K)=\{1,\gamma\}$.
Define a cocycle (homomorphism)
$$
c\colon\Gamma\to\SS_3\times\SS_2
$$
taking $\gamma$ to the nontrivial element $\ve\in \SS_2$.
We obtain a twisted torus $_c T$.
Let $_cT_L$ denote the corresponding $(\SS_3\times\SS_2,\Gamma)$-variety over $L$,
it is $T_L:=T\times_K L$ with the following actions:
\begin{align}
&\sigma(x_1, x_2, x_3):= (x_{\sigma^{-1}(1)}, x_{\sigma^{-1}(2)}, x_{\sigma^{-1}(3)}) \quad\text{for}\hskip 2mm \sigma\in\Sym_3,\\
&\ve(x_1,x_2,x_3)=(x_1^{-1},x_2^{-1},x_3^{-1}),\\
&^\gamma(x_1,x_2,x_3)=(\ ^\gamma x_1^{-1},\ ^\gamma x_2^{-1},\ ^\gamma x_3^{-1}).\label{eq:Gamma-tw}
\end{align}
Note that $_c(\SS_3\times\SS_2)=\SS_3\times\SS_2$, because $c(\gamma)=\ve$ is central in $\SS_3\times\SS_2$.
\end{subsec}

\begin{proposition}\label{prop:S3S2}
There exists a birational $(\SS_3\times\SS_2,\Gamma)$-isomorphism between the  $(\SS_3\times\SS_2,\Gamma)$-varieties
$_c T_L\times_L\G_{m,L}^2$ and $\Lie(_c T_L)\times_L\A_L^2$.
\end{proposition}

\begin{proof}
This follows from Proposition \ref{prop:LPR-G2} and Lemma \ref{lem5.4(a)}.
\end{proof}

\begin{subsec}
We define two embeddings $\SS_3\into\SS_3\times\SS_2$, the standard one and the twisted one:
\begin{align*}
&\St(\sigma)=(\sigma,1) \text{ for } \sigma\in \SS_3,\\
&\Tw(\sigma)=(\sigma,\ve^{\sign(\sigma)})=
\begin{cases}
(\sigma,1)&
 \text{if \ $\sign(\sigma) =1$, } \\
(\sigma,\ve) &
\text{if \ $\sign(\sigma) =-1$.}
\end{cases}
\end{align*}
These two embeddings define two $\SS_3$-actions on $_c T_L$.
We denote the corresponding $(\SS_3,\Gamma)$-varieties (with the twisted $\Gamma$-action \eqref{eq:Gamma-tw})
by $_\St T'_L$ and $_\Tw T'_L$, respectively.
\end{subsec}

\begin{corollary}\label{cor:St-Tw}
There exist birational $(\SS_3,\Gamma)$-isomorphisms
$$
_\St T'_L\times_L \G_{m,L}^2\dasheq\Lie(_\St T'_L)\times_L\A_L^2 \text{ and }
_\Tw T'_L\times_L \G_{m,L}^2\dasheq\Lie(_\Tw T'_L)\times_L\A_L^2 .
$$
\end{corollary}

\begin{proof}
The $(\SS_3\times\SS_2,\Gamma)$-equivariant birational isomorphism of Proposition \ref{prop:S3S2}
is, in particular, $(\SS_3,\Gamma)$-equivariant with respect to each of the two embeddings $\St,\Tw\colon \SS_3\into\SS_3\times\SS_2$.
\end{proof}

\begin{subsec}\label{subsec:L,K,H}
Let $L/K$ be an arbitrary quadratic extension of fields of characteristic 0.
Let $G=\SU(3, L/K,H):=\SU(L^3, H)$, the special unitary group of the $L/K$-Hermitian form with  matrix $H$,
where $H\in M_3(L)$ is a nondegenerate  Hermitian matrix.
Then $G$ is a simple $K$-group, an outer $L/K$-form of the split $K$-group $\SL_{3,K}$.
Note that  $G=\SU(3,L/K,H)$ is an {\em inner} form of the $K$-group
$\SU_{3,L/K}:=\SU(3,L/K, I_3)$, where $I_3=\diag(1,1,1)$.
\end{subsec}

\begin{proposition}\label{prop:SU3-Gm2}
Let a quadratic extension $L/K$ and a Hermitian matrix $H\in M_3(L)$ be as in \S\,\ref{subsec:L,K,H}.
Let  $G=\SU(3,L/K,H)$, then $G\times_K \G_{m,K}^2$ is Cayley.
\end{proposition}

\begin{proof}
Since $G$ is an inner form of $\SU_3:=\SU(3,L/K, I_3)$, by Proposition \ref{rem2.4} it suffices to consider the case of $\SU_3$.
Let $T_{\SU_3}$ denote the diagonal maximal torus of $\SU_3$,
we can identify it with the torus $_\St T'_L$ of Corollary \ref{cor:St-Tw}.
Now our proposition follows from Corollary \ref{cor:St-Tw} and Proposition \ref{prop:cor6.5}.
\end{proof}

\begin{proposition}\label{prop:PGU3-Gm2}
Let a quadratic extension $L/K$ and a Hermitian matrix $H\in M_3(L)$ be as in \S\,\ref{subsec:L,K,H}.
Let $G=\gPGU(3,L/K,H)$ be the adjoint $K$-group corresponding to the simply connected $K$-group $\SU(3, L/K,H)$.
Then $G\times_K \G_{m,K}^2$ is Cayley.
\end{proposition}

\begin{proof}
Since $G$ is an inner form of $\gPGU_3:=\gPGU(3,L/K,I_3)$, by Proposition \ref{rem2.4} it suffices to consider the case of $\gPGU_3$.
Let $T_{\PGU_3}\subset\PGU_3$ denote the image of the diagonal maximal torus of $\SU_3$,
we can identify the corresponding $L$-torus  $T_{\PGU_3}\times_K L$ with
the torus $(\G_{m,L}^3/\G_{m,L})_{\Gamma\text{-twisted}}$ endowed with the following actions of $\SS_3$ and $\Gamma$:
\begin{align*}
&\sigma([x_1, x_2, x_3]):= [x_{\sigma^{-1}(1)},\ x_{\sigma^{-1}(2)},\ x_{\sigma^{-1}(3)}] \quad\text{for}\hskip 2mm \sigma\in\Sym_3\\
&^\gamma[x_1, x_2,x_3]=[\ ^\gamma x_1^{-1},\ ^\gamma x_2^{-1},\ ^\gamma x_3^{-1}].
\end{align*}
We define a homomorphism  $\G_{m,L}^3/\G_{m,L}\to T_L$ by
$$
\quad [x_1,x_2,x_3]\mapsto (x_2/x_3, x_3/x_1, x_1/x_2).
$$
One checks immediately that we obtain an $(\SS_3,\Gamma)$-equivariant isomorphism
$$
(\G_{m,L}^3/\G_{m,L})_{\Gamma\text{-twisted}}\isoto\  _\Tw T'_L.
$$
and its differential, which is also  an $(\SS_3,\Gamma)$-equivariant isomorphism,
$$
\Lie\,(\G_{m,L}^3/\G_{m,L})_{\Gamma\text{-twisted}}\isoto\Lie\,_\Tw T'_L.
$$
By Corollary \ref{cor:St-Tw} there exists an $(\SS_3,\Gamma)$-equivariant birational isomorphism
$$
 _\Tw T'_L\times_L\G_{m,L}^2 \dasheq\Lie\,_\Tw T'_L\times_L\A_L^2.
$$
Combining these birational isomorphisms, we obtain an $(\SS_3,\Gamma)$-equivariant birational isomorphism
$$
(\G_{m,L}^3/\G_{m,L})_{\Gamma\text{-twisted}}\times_L\G_{m,L}^2\dasheq \Lie\ (\G_{m,L}^3/\G_{m,L})_{\Gamma\text{-twisted}}\times_L\A_{L}^2,
$$
that is, an $\SS_3$-equivariant birational isomorphism
$$T_{\PGU_3}\times_K\G_{m,K}^2 \dasheq\Lie(T_{\PGU_3})\times_K\A_K^2.$$
Now Proposition \ref{prop:PGU3-Gm2} follows from Proposition \ref{prop:cor6.5}.
\end{proof}

\begin{proof}[Proof of Theorem \ref{thm:2}]
If $G$ is of absolute rank 1, then by Proposition \ref{pr1} the group $G$ is Cayley
(and hence, the group $G\times_K \G_{m,K}^2$ is Cayley).
Now assume that $G$ is of absolute rank 2.
If $G$ is not semisimple, or  is of type $\AA_1\times\AA_1$, or is of type $\BB_2=\CC_2$,
then by Propositions \ref{pr2}, \ref{pr3}, and \ref{pr4} the group $G$ is Cayley, hence the group $G\times_K \G_{m,K}^2$ is Cayley.
Otherwise $G$ is of type $\GG_2$ or $\AA_2$, and by Example \ref{ex2} and Propositions
 \ref{prop:G2}, \ref{prop:SL3},
 \ref{prop:SU3-Gm2}, and  \ref{prop:PGU3-Gm2}  the group $G\times_K \G_{m,K}^2$ is Cayley.
\end{proof}

\bigskip\bigskip\bigskip\bigskip\bigskip

\begin{center}
{ \bf Appendix A\\ \bigskip \ ELEMENTARY LINKS}
\bigskip

{\bf by Igor Dolgachev}

\end{center}
\bigskip

In this appendix we will follow the ideas from Iskovskikh's papers \cite{Isk1}, \cite{Isk2}, \cite{Isk3} to study the Cayley property of the
groups $\SU_3,\ \PSU_3$  and $\SL_3$ over $\bbR$.

\section{Elementary links for $G$-surfaces}
Let $X$ be a smooth projective surface over a perfect field $K$ and $G$ be a finite group of $K$-automorphisms of $X$.
We say that the pair $(X,G)$ is a $G$-surface. Two $G$-surfaces $(X,G)$ and $(X',G)$ are called birationally (biregularly) isomorphic
if there exists a birational (biregular) $G$-equivariant
map $\phi:X\da X'$ defined over $K$.
A $G$-surface $(X,G)$ is called \emph{minimal} if any birational $G$-equivariant morphism $X\to X'$ is an isomorphism.
Any birational $G$-map between two $G$-surfaces can be factored into a sequence of birational $G$-morphisms and their inverses.
A birational $G$-morphism $f:X\to Y$ is isomorphic to the blow-up of a closed $G$-invariant $0$-dimensional subscheme $\fraka$ of $Y$.
For the future use let us remind that the degree of $\fraka$ is the number $\deg(\fraka) = h^0(\calO_\fraka)$.
If $\fraka$ is reduced and consists of closed points
$y_1,\ldots,y_k$ with residue fields $\kappa(y_i)$, then $\deg(\fraka) = \sum \deg(y_i)$, where $\deg(y_i) = [\kappa(x_i):K]$.
The $G$-invariance of $\fraka$ means that $\fraka$ is the union of $G$-orbits.

The birational classification of $G$-surfaces over $K$ is equivalent to the classification of minimal $G$-surfaces up to birational isomorphisms.

From now on we assume that $X$ is a rational surface, i.e. after a finite base change $L/K$, the surface is birationally isomorphic  to $\bbP_L^2$.
It is known (see \cite{Isk1}) that a minimal rational surface belongs to one of the following two classes:
\begin{itemize}
\item[($\calD$)] $X$ is a del Pezzo surface with $\Pic(X)^G \cong \bbZ$;
\item[($\calC$)] $X$ is a conic bundle with $\Pic(X)^G\cong \bbZ^2$.
\end{itemize}

Recall that $X$ is called a \emph{del Pezzo surface} if the anti-canonical sheaf $\omega_X^{-1}$ is ample.
The self-intersection number $(\omega_X,\omega_X)$ takes its value between $1$ and $9$ and is called the \emph{degree} of a del Pezzo surface.  Also  $X$ is a called a
\emph{conic bundle} if there exists a $K$-morphism  $f:X\to C$ such that each fiber is reduced and is isomorphic to a conic over $K$ (maybe reducible).

In the case when $K$ is an algebraically closed field, the problem of birational classification of minimal $G$-surfaces
is equivalent to the problem of classification of conjugacy classes of finite subgroups of the Cremona group $\textup{Cr}_K(2)$
of birational automorphisms of $\bbP_K^2$.  We refer to \cite{DolIsk} for  the results in this direction.
When $G= \{1\}$, the problem of classification of rational $K$-surfaces
has been addressed in fundamental works of V.A.~Iskovskikh \cite{Isk} and Yu.I.~Manin \cite{Manin}.
In both cases a modern approach uses the theory of elementary links \cite{Isk1}.

We will be dealing with minimal del Pezzo $G$-surfaces or minimal conic bundles $G$-surfaces.
In the $G$-equivariant version of the Mori theory they are interpreted as extremal contractions $\phi:S\to C$, where $C = \textup{pt}$
is a point in the first case and $C$ is a curve in the second case. They are also two-dimensional analogs of rational Mori $G$-fibrations.

A birational $G$-map between Mori fibrations is a diagram of  $G$-equivariant rational $K$-maps
\begin{equation}\label{link1}
\xymatrix{S{}\ar[d]_\phi\ar@{-->}[r]^f&S'\ar[d]_{\phi'}\\
C&C'}
\end{equation}
which in general do not commute with the fibrations.
Such a map is decomposed into \emph{elementary links}.
These links are divided into the four  following types.

\begin{itemize}
\item Links of type I:
\end{itemize}
They are commutative diagrams of the form
\begin{equation}\label{link2}
\xymatrix{S{}\ar[d]_\phi&Z=S'\ar[d]_{\phi'}\ar[l]_\sigma\\
C=\textup{pt}&C'= \bbP^1\ar[l]_\alpha}
\end{equation}
Here $\sigma:Z\to S$ is a  blow-up of a closed $G$-invariant $0$-dimensional subscheme $G$-orbit, $S$ is a minimal Del Pezzo surface,
$\phi':S'\to \bbP^1$ is a minimal conic bundle, $\alpha$ is the constant map.
For example, the blow-up of a $G$-fixed $K$-rational point on $\bbP^2$ defines a minimal conic $G$-bundle
$\phi':\bfF_1\to \bbP^1$ with a $G$-invariant exceptional section. Here and in the sequel we denote by
$\bfF_n$ a $K$-surface which becomes isomorphic over the algebraic closure of $K$ to a minimal ruled surface
$\bbP(\calO_{\bbP_K^1}\oplus \calO_{\bbP_K^1}(-n)).$

\begin{itemize}
\item Links of type II:
\end{itemize}
They are commutative diagrams of the form
\begin{equation}\label{link3}
\xymatrix{S{}\ar[d]_\phi&Z\ar[l]_\sigma\ar[r]^\tau&S'\ar[d]_{\phi'}\\
C&=&C'}
\end{equation}
Here $\sigma:Z\to S, \tau:Z\to S'$ are the blow-ups of $G$-invariant closed $0$-dimensional subschemes such that
$\rank~\Pic(Z)^G = \rank~\Pic(S)^G+1 = \rank~\Pic(S')^G+1$, $C=C' $ is either a point or a curve. An example of a link of type II is
the a link between $\bbP^2$ and $\bbF_0$ where  one blows up a $G$-invariant closed subscheme $\fraka$ of $\bbP_K^2$ of degree 2
and then blows down the proper transform of the line spanned by $\fraka$. Another frequently used link of type II
is an elementary transformation of minimal ruled surfaces and conic bundles.

\begin{itemize}
\item Links of type III:
\end{itemize}
These are the birational maps which are the inverses of links of type I.

\begin{itemize}
\item Links of type IV:
\end{itemize}
They exist when $S$ has two different structures of $G$-equivariant conic bundles. The link is the exchange of the two conic bundle structures
\begin{equation}\label{link4}
\xymatrix{S{}\ar[d]_\phi&=&S'\ar[d]_{\phi'}\\
C&&C'}
\end{equation}

\begin{theorem}\label{fact}
Let $f:S\dashrightarrow S'$ be a birational map of minimal $G$-surfaces.
Then $f$  is equal to a composition of $G$-equivariant elementary links.
\end{theorem}

The proof of this theorem is the same as in the arithmetic case considered in \cite{Isk2}, Theorem 2.5.

To start an elementary link, one has to blow up a $G$-invariant subscheme of maximal multiplicity of a linear system defining the birational map.

The classification of possible elementary links can be found in \cite{Isk1}. It is stated in the case $G=\{1\}$, however it can be extended
to the general case in a straightforward fashion. The case when $G\ne \{1\}$ but $K$ is algebraically closed is considered in \cite{DolIsk}, 7.2.

\begin{example}\label{ex1} Assume  $X$ is a del Pezzo surface $\calD_6$ of degree $6$ and $X' = \bbP_K^2$.
We want to decompose a birational $G$-equivariant map $X\da X'$ into a composition of elementary links.
From Propositions 7.12 and 7.13 in \cite{DolIsk} we obtain that the only elementary link starting at $(X,G)$ ends
either at a del Pezzo surface $Y$ of degree 6 or at $\bfF_0$. Since we do not want to stay on some
$(\calD_6,G)$,  we may assume that $Y = \bfF_0$. Now we need an elementary link starting at $Y$. The same Propositions tell us that the end
of the next elementary link is either a conic bundle $Y'\to C$, or $\bfF_0$, or $\bbP_K^2$, or a del Pezzo surface of degree 5 or 6.
Since we do not want to return back to $X$ or $\bfF_0$ we may assume that the end of the link $Y'$ is either a conic bundle
or a del Pezzo surface of degree 5, or $\bbP^2$.
If $Y' = \bbP^2$, then Proposition 7.13, case 2, tells us that $\bfF_0$ must contain a $G$-invariant $K$-rational point.
If $Y'$ is a del Pezzo surface of degree 5, then the same Proposition tells us that $Z\to Y'$ is the blow-up of a $G$-invariant subscheme of degree 5.
Finally, if $Y'$ is a conic bundle, we may continue to do elementary links staying in the class $\calC$ and at some point we have
to link a conic bundle with a del Pezzo surface $Y''$.
Proposition 7.12 tells us that $Y''$ is either a del Pezzo surface of degree 4 or $\bfF_0$.
Since we do not want to return back to $\bfF_0$, we may assume that $Y''$ is a del Pezzo surface of degree 4.
However, we find from Proposition 7.13, case 5, that we are stuck here since any elementary link  relates $Y''$  only with itself.

 Assume $X$ is birationally $G$-isomorphic to $\bbP_K^2$. Then the previous analysis  shows that  $X$ must have  a $G$-invariant rational
 $K$-point allowing us to find an elementary link with $\bfF_0$.  To continue, we need to find either a $K$-rational $G$-equivariant point on $\bfF_0$
to link the latter with $\bbP_K^2$, or to find  a $G$-invariant $0$-dimensional subscheme of length 5 to link $\bfF_0$ with a del Pezzo surface $\calD_5$ of degree 5.
The only elementary link which ends not at a del Pezzo surface of degree 5 or $\bfF_0$ is a link
 connecting to $\bbP_K^2$. It follows from Proposition 7.13, case 4, that to perform this link we need a $K$-rational $G$-invariant point on $\calD_5$.

Here we exhibit possible elementary links relating a del Pezzo $G$-surface $(\calD_6,G)$  with $(\bbP^2,G)$.

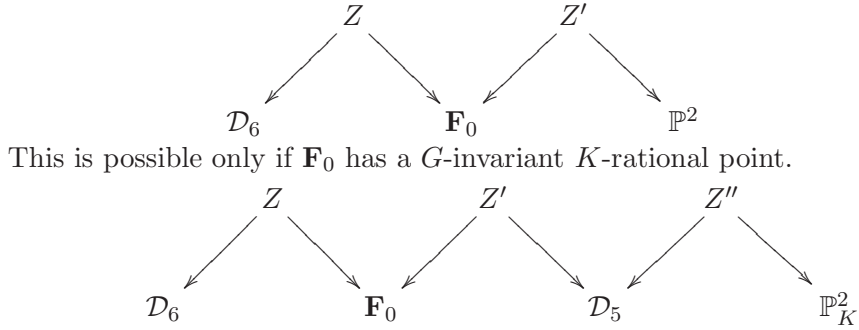
\xymatrix{&&&&Z\ar[dl]\ar[dr]&&Z'\ar[dl]\ar[dr]\\
&&&\calD_6&&\bfF_0&&\bbP^2}

This is possible only if $\bfF_0$ has a $G$-invariant $K$-rational point.

\xymatrix{&&&Z\ar[dl]\ar[dr]&&Z'\ar[dl]\ar[dr]&&Z''\ar[dl]\ar[dr]\\
&&\calD_6&&\bfF_0&&\calD_5&&\bbP_K^2}

This is possible only if $\bfF_0$  has a $G$-invariant closed subscheme of degree $5$, and also $\calD_5$ has a $K$-rational $G$-invariant point.
\end{example}

\section{Maximal tori in $\SU(3),\PSU(3)$}
Let $\SL_3$ be the split simply connected  simple group of type $\AA_2$ over the field of real numbers.  Let $\SU_3$ be its real form
defined by the element of $H^1(\Gal(\bbC/\bbR),\SL_3(\bbC))$ represented by the map $A\mapsto \bar{A}^{-1}$.
Its group of real points $\SU_3(\bbR)$ is isomorphic to the group $\textup{SU}(3)$ of unitary $3\times 3$ complex matrices.
A maximal torus $\bbT$ in $\SU_3$ is a real form of the standard
torus $(\bbC^*)^2 = \{(z_1,z_2,z_3)\in (\bbC^*)^3:z_1z_2z_3 = 1\}.$ It is defined by the map $(z_1,z_2,z_3)\mapsto
(\bar{z}_1^{-1},\bar{z}_2^{-1},\bar{z}_3^{-1})$ and it is isomorphic to $(\bbS^1)^2$, where
$\bbS^1 = \Spec~\bbR[x,y]/(x^2+y^2-1)$ with the natural structure of an algebraic  group  over $\bbR$.
The group of real points of $\bbS^1$ is the circle $\textup{SU}(1) = \{z\in \bbC: |z| = 1\}$.
 Its complex points are
$\{(z_1,z_2)\in \bbC^2:z_1^2+z_2^2 = 1\}$. The isomorphism $\bbS^1(\bbC) \to \bbC^*$ is given by $(z_1,z_2)\mapsto z = z_1+iz_2$.

Let $C = \Proj~\bbR[t_0,t_1,t_2]/(t_1^2+t_2^2-t_0^2)$ be the standard compactification of $\bbS^1$.
It is a plane nonsingular conic  defined over $\bbR$. Its real points satisfying $t_0 \ne 0$ are identified with
$\textup{SU}(1)$ via the map $a+bi\mapsto [a,b,1]$. Let
$$f:\mathbb{P}^1\to C, \quad [u,v]\mapsto [u^2-v^2,2uv,u^2+v^2]$$
be the rational parameterization of $\bbS^1$ defined over $\bbR$. We have
$$[u,v]\cdot [u',v']:= [uu'-vv',uv'+u'v]$$ is mapped to
{\small $$[(uu'-vv')^2-(uv'+u'v)^2,2(uu'-vv')(uv'+u'v),(uu'-vv')^2+(uv'+u'v)^2] =$$
$$= [(u^2-v^2)(u'{}^2-v'{}^2)-4uvu'v',(u^2-v^2)2u'v'+(u'{}^2-v'{}^2)2uv,(u^2+v^2)(u'{}^2+v'{}^2)].$$}
This shows that the restriction of the map $f$ to  the open subset $D^+(u^2+v^2)$  is a homomorphism of groups.

Now let us consider the subvariety $X$ of $(\mathbb{P}^1)^3$ given by the condition that $x\cdot y\cdot z = (1,0)$. It is given by the equation
$$uu'v''-vv'v''+uv'u''+u'vu'' = 0.$$
This is a compactification of the maximal torus $\bbT$ in $\SU_3$. The equation is given by a trilinear function, hence $X$ is a hypersurface in
$(\bbP_K^1)^3$ of type $(1,1,1)$. By the adjunction formula,
$$K_X = (K_{(\mathbb{P}^1)^3}+X)\cdot X  = -(h_1+h_2+h_3).$$
This shows that $X$ is a del Pezzo surface, anticanonically embedded in $\mathbb{P}^7$ by means of the Segre map
$(\mathbb{P}^1)^3\hookrightarrow \mathbb{P}^7$. Here $h_i$ are the preimages of $\mathcal{O}_{\mathbb{P}^1}(1)$ under the projections $p_i:X\to \mathbb{P}^1$.
The degree of the del Pezzo surface $X$  is equal to $(h_1+h_2+h_3)^3 = 6h_1h_2h_3 = 6$.
Over $\bbC$, a del Pezzo surface of degree 6 is isomorphic to the blow-up of three non-collinear points in $\bbP^2$.

The boundary $X\setminus \bbT$ of the torus $\bbT$ consists of three irreducible (over $\bbR$) components $p_i^{-1}(V(u^2+v^2))$.
Over $\bbC$, each such component splits into two disjoints curves isomorphic to $\bbP^1$. The boundary becomes a hexagon of lines in the
anticanonical embedding. The opposite sides are the pairs of conjugate lines.
The group of automorphisms of the root system of type $\AA_2$ of the group $\SU_3$ is isomorphic to the dihedral group $D_{6}$ of order 12
(also isomorphic to the direct product $\frakS_3\times \bbZ/2\bbZ$). Its standard action on $\bbT$ extends to a faithful action on the
compactification $X$.  It acts on the hexagon via its obvious symmetries.

Note that the Picard group $\Pic(X_\bbC)$ is generated by the classes $e_0,e_1,e_2,e_3$, where $e_0$
is the class of the preimage of a line under the blow-up $X_\bbC = X_\bbC\to \bbP_\bbC^2$, and $e_i$ are the classes of the exceptional curves.
The hexagon of lines on $X$ consists of the six lines with the divisor classes
$$e_1,\ e_2,\ e_3,\  f_1= e_0-e_2-e_3,\  f_2 = e_0-e_1-e_3,\  f_3 = e_0-e_1-e_2.$$
The pairs of
opposite sides are $\{f_i,e_i\}$. The group $\frakS_3$ acts on $\Pic(X)$ by permuting $e_1,e_2,e_3$, and the Galois group acts on $\Pic(X)$ by $f_i\mapsto e_i$. Note that
$-K_X = 3e_0-e_1-e_2-e_3$ and, since $K_X$ is Galois invariant, the conjugation isometry of $\Pic(X_\bbC)$  sends $e_0$ to $2e_0-e_1-e_2-e_3 = -K_X-e_0$ and
$e_0-e_i$ to $-K_X-e_0-(e_0-e_j-e_k) = e_0-e_i$. This shows that the pencil of conics $|e_0-e_i|$ defines a map $p_i:X\to \bbP^1$ over $\bbR$. This defines our embedding
$$X\hookrightarrow (\bbP^1)^3\hookrightarrow \bbP^7.$$
Also note that the invariant part $\Pic(X)^{\frakS_3\times \Gal(\bbC/\bbR)} = \bbZ K_X$, i.e. $X$ is a minimal $\frakS_3$-surface over
$\bbR$.

Consider the real point $e\in \bbT(\bbR)$, the unit element of the torus.
The tangent plane to the Segre variety $s(\bbP^1\times \bbP^1\times \bbP^1)$ in $\bbP^7$ is spanned by the images
of $e\times \bbP^1\times \bbP^1,\ \bbP^1\times e\times \bbP^1$ and $\bbP^1\times \bbP^1\times e$.
Its intersection with $X$ is the point $e$.
Consider the projection $\bbP^7\da \bbP^3$ from the tangent plane of $X$ at $e$.
Its restriction to $X$ defines a rational map $X\dasharrow Q$, where $Q$ is
 a nonsingular  quadric $Q$ in $\bbP^3$. In fact, the rational map is the composition $\tau\circ \pi^{-1}$, where $\pi:X'\to X$ is the blow-up of the point $e$,
and $\tau:X'\to Q$ is the blow-down of the proper transforms of three conics $R_i$, the images of
$(\bbP^1\times \bbP^1\times \{e\})\cap X,\ (\{e\}\times \bbP^1 \times \bbP^1)\cap X,$ and $(\bbP^1\times \{e\}\times \bbP^1)\cap X$.
Note that, $R_i^2 = 0$ on $X$, and $\bar{R}_i^2 = -1$ on $X'$. We have $K_{X'}^2 = K_X^2-1 = 6-1 = 5$, and $K_Q^2 = 5+3 = 8$,
so $Q$ is a del Pezzo of degree 8, i.e a quadric or  $\mathbf{F}_1$. But the latter is not embedded in $\bbP^3$ as a normal surface.

The surface $X$ has three $\frakS_3$-invariant  points $e,\eta,\eta^2\in \bbT(\bbR)$ corresponding to the diagonal matrices in $\textup{SU}(3)$. The image of
$\eta$ is $\frakS_3$-invariant which is a real point in the real structure of $Q$ defined by the map $X\to Q$.
Projecting from this point, we see that $Q$ is birationally trivial over $\bbR$ as a $\frakS_3$-surface.

Applying Proposition \ref{prop:cor6.5}, we obtain

\begin{theorem} \label{thm:su3-Igor}
The group $\SU_3$ is a Cayley group.
\end{theorem}

Next we consider the group $\PSU(3)$.
It is the quotient of $\SU_3$ by the cyclic group $\mu_3$ of order 3. Its group $\PSU_3(\bbR)$ of real points is isomorphic to the
group $\textup{PSU}(3)$. A maximal torus of $\PSU_3$ is isomorphic to $\bbT/\mu_3$, where $\bbT$ is a maximal torus of $\SU_3$.
In the real picture from the previous section, its action on $\bbT$ is the multiplication map
$\sigma:(u,v) \mapsto (u,v)\cdot (1/2,\sqrt{3}/2).$
The action of $\mu_3$ extends to the compactification $X$ of the maximal torus $\bbT$ of $\SU_3$.
Obviously, it  leaves invariant the boundary, and has six isolated fixed points on the boundary; they are the vertices of the hexagon.
The automorphism group of the del Pezzo surface $X$ (over $\bbC$) is $(\bbC^*)^2\times D_6$,
and $\sigma$ belongs to the connected part, and hence acts identically on $\Pic(X)$.
In particular, it acts identically on the  sides of the hexagon of lines. The quotient  $Y = X/\mu_3$ is a singular compactification
of a maximal torus of $\PSU_3$.
It has six singular quotient singularities  of type $\frac{1}{3}(1,1)$, a minimal resolution $Y'\to Y$
has six exceptional curves $E_i$ with $E_i^2 = -3$.
The proper transforms of the images of the sides of the hexagon are six disjoint $(-1)$ curves\footnote{a $(-n)$-curve is a smooth rational curve
on a nonsingular projective surface with self-intersection equal to $-n$.}.
Together with $E_i$'s they form a 12-gon\footnote{One can also arrive at this 12-gon by first blowing up the vertices of the hexagon,
then extend the action of $\mu_3$ to the blow-up, and then taking the quotient.}.
 All of this is defined over $\bbR$, the Galois group switches opposite
$(-3)$-sides and opposite $(-1)$-sides of the 12-gon.
Now we can blow down the $(-1)$-sides to get a nonsingular surface $Z$ with a hexagon of $(-1)$-curves formed by the images of the $(-3)$-sides.
So, $Z$ is a del Pezzo surface of degree six again!
 We have found a nonsingular $\frakS_3\times \Gal(\bbC/\bbR)$-invariant minimal compactification
of a maximal torus of  $\PSU_3$ which is a del Pezzo surface of degree six.

Note that the group $\frakS_3\times \Gal(\bbC/\bbR)$ acts on $\Pic(Z)$ in the same way as it acts in the  case of $\SU_3$.
So, as before, we have a $\frakS_3$-invariant  embedding $Z\hookrightarrow \bbP^7$ defined over $\bbR$
with a rational point equal to the orbit $\bar{e}$ of the origin $e\in X$ which consists of the diagonal matrices of $\text{SU}(3)$.
This time we have no any other $\frakS_3$-invariant rational points on $X$ (they obviously do not lie on the boundary).
By projection from the point $\bar{e}$, we obtain a quadric $Q$.

The projection defines a $\frakS_3$-equivariant isomorphism over $\bbR$ between the complement of the three conics on $X$
and the complement of the image of the exceptional curve over $\bar{e}$ in $Q$. The latter curve is a conic section $R'$ of $Q$.
The three conics are permuted under $\frakS_3$, so $\frakS_3$ acts on $R'$ without fixed points.
Thus a $\frakS_3$-invariant real point on $Q$ must be the projection of a real $\frakS_3$-invariant point on the del Pezzo surface $X$.
There is none except the point which has been blown up.
Thus the quadric $Q$ has no $\frakS_3$-invariant real points. It follows from
Example \ref{ex1} that there is no  birational $\frakS_3$-equivariant map from $Z$ to $\bbP_\bbR^2$ (we are stuck at the first elementary link!).

Using Proposition \ref{prop:cor6.5}, we obtain

\begin{theorem}\label{thm:PGU3}
 The group $\PSU_3$ is not Cayley.
\end{theorem}

\section{Maximal tori in $\SL_3$}

The group $\SL_3$ is a  simple algebraic group split over $\bbR$. Its group of real points $\SL_3(\bbR)$ is the group  of unimodular real
$3\times 3$-matrices. Its maximal torus is the standard torus
$\bbT = \Spec~\bbR[z_1,z_2,z_3]/(z_1z_2z_3-1)$. The group $\bbT(\bbR)$ of its real points is naturally isomorphic to
$\{(a,b,c)\in (\bbR^*)^3:abc = 1\}$ with the $\frakS_3$-action defined by permutation of the coordinates.
Obviously, a real $\frakS_3$-invariant point on $\bbT$ must be equal to the identity point $(1,1,1)$.

A natural $\bbT$-equivariant compactification of $\bbT$ is the cubic surface
$Y = \Proj~\bbR[t_0,t_1,t_2,t_3]/(t_1t_2t_3-t_0^3)$. It has three quotient singularities of type $\frac{1}{3}(1,2)$, rational double points of type $\AA_2$.
They are defined over $\bbR$. The exceptional curve over each singular point consists of two $(-2)$-curves $E_i+E_i'$ intersecting transversally at one point.
The intersection point $E_i\cap E_i'$ is a real point, hence the curves are isomorphic to $\bbP^1$ over $\bbR$. The group $\frakS_3$ permutes the pairs $(E_i,E_i')$.
After we minimally resolve $Y$ over $\bbR$, we obtain a surface isomorphic to the
blow-up of a del Pezzo surface of degree 6 at three vertices of the hexagon of lines. The boundary consists of a 9-gon with 9 consecutive
sides $R_1,\ldots,R_9$, where $R_1,P_2,R_4,R_5,R_7,R_8$ are $(-2)$-curves and the sides
$R_3,R_6,R_9$ are $(-1)$-curves. The latter curves are the proper transforms of the three lines on the cubic surface $Y$ that join the pairs of the singular points.
After we blow down ($\frakS_3$-equivariantly) the (-1) curves, we obtain a del Pezzo surface $X$ of degree 6 with a hexagon of lines at the boundary.
The linear system that defines the rational map $Y\da X$ consists of quadric sections of $Y$ passing through the singular point.
Note that both $X$ and $Y$ are $\frakS_3$-equivariant compactifications of $\bbT$.

\begin{theorem}\label{thm:SL3}
 $\SL_3$ is not Cayley.
\end{theorem}

\begin{proof} By Proposition \ref{prop:cor6.5} it suffices to prove that  $(X,\frakS_3)$ is not birationally isomorphic to a
 $(\bbP_\bbR^2,\frakS_3)$. Suppose they are birationally isomorphic. It follows from Example \ref{ex1} that the first link must end at
 $\bfF_0\cong \bbP_\bbR^1\times \bbP_\bbR^1$ which we identify with a split nonsingular quadric $Q$ in $\bbP_\bbR^3$.
The link  consists of blowing up the unique $\frakS_3$-invariant real point on $X$, namely the point $e$, and then blowing down three $(-1)$-curves.
They are the images of the conics on $Y$ that, together with the three lines, are cut out by the quadrics $t_it_j - w^2 = 0$.
The conics are left invariant under the conjugation but permuted by $\frakS_3$. The action of $\frakS_3$ on $X$ shows easily
that the induced action of $\frakS_3$ on $Q$ permutes the two rulings (i.e. the two projections to $\bbP^1$).
It is easy to see, using the description of automorphisms of $\bbP_\bbR^1\times \bbP_\bbR^1$, that the quadric
 $Q$ has no real $\frakS_3$-invariant points, so the next elementary link relates $Q$ with a del Pezzo surface $\calD_5$ of degree 5. For
 this we need a $\frakS_3$-invariant $0$-dimensional subscheme $\fraka$ of degree 5.  It must consist of a $\frakS_3$-invariant point of
 degree 2 and an $\frakS_3$-orbit of three real points. It is easy to see that the only $\frakS_3$-invariant point of
 degree 2 is the image of two conjugate scalar matrices in $\SL_3(\bbC)$. There are plenty of $\frakS_3$-orbits of three real points.
 Now we have to  apply the elementary link $Q\leftarrow Z\rightarrow \calD_5$ with the target equal to a del Pezzo surface $\calD_5$ of degree 5.
Either we are stuck here and hence prove the assertion or we find a real $\frakS_3$-invariant point on $\calD_5$ to make the final elementary link with
$(\bbP^2,\frakS_3)$. Since $Q$ has no such points, a real $\frakS_3$-invariant point $q$ on $\calD_5$ lies on the image
of  an exceptional curve of $Z\to Q$ or on the image of  an exceptional curve of $Z\to \calD_5$.
The three exceptional curves on $Z$ over real points in $Q$ are permuted by
$\frakS_3$, so $q$ cannot lie on them. Also the exceptional curve on $Z$ over the complex point in $Q$ consists of two disjoint conjugate curves.
So, $q$ is not on them either. It follows from the description of the linear system defining the link, that the exceptional curves of
$Z\to \calD_5$ are the proper transforms $\bar{R}_1$ and $\bar{R}_2$ of the two rational curves $R_1$ and $R_2$ of degree 3 (of bidegrees $(2,1)$ and $(1,2)$) on $Q$.
Since $\frakS_3$ permutes the two rulings on $Q$, it cannot leave $R_1$ or $R_2$ invariant.
Thus the images of the exceptional curves $\bar{R}_1$ and $\bar{R}_2$ are not fixed under $\frakS_3$.
Thus the point $q$ cannot be one of these points. This shows that the last elementary link
$\calD_5\da \bbP^2$ is not possible.
\end{proof}

\begin{remark} The real split group $\PSL_3$ is known to be a Cayley group (see \cite[Example 1.11]{LPR}). Using Proposition \ref{prop:cor6.5},
 this fact immediately follows from the existence of a $\frakS_3$-equivariant compactification of a maximal torus of $\PSL_3$ isomorphic to the projective plane.
In fact, consider the cubic surface $X$ from the proof of the previous theorem. The quotient of this surface by
the cyclic group generated by the transformation $[t_0,t_1,t_2,t_3]\mapsto [\eta_3t_0,t_1,t_2,t_3]$ is isomorphic to $\bbP_\bbR^2$ via the
projection map from the point $[1,0,0,0]\in \bbP^3\setminus X$.
 Its maximal  torus $\bbT$ is the standard torus in $\bbP_\bbR^2$.
\end{remark}

\bigskip\bigskip\bigskip\bigskip\bigskip

\appendix
\setcounter{section}{2}
\setcounter{lemma}{0}

\begin{center}
{ \bf Appendix B\\ \bigskip \ BAD CHARACTERISTICS}
\bigskip
\end{center}
This appendix was contributed by the anonymous referee.
Since the referee's original exposition has been  changed,
the responsibility  for possible inaccuracies or mistakes lies on the author of the paper.

\begin{theorem}\label{thm:referee}
Let $K$ be a field of characteristic $p>0$.
We write $\G_m$ for the multiplicative group $\G_{m,K}$, and $\G_a$ for the additive group $\G_{a,K}$.
Let $A$ be a central simple algebra of degree $n$ over $K$.
Assume that $p|n$ and that $4|n$ if $p=2$.
Then the group $G=\PGL_1(A):=A^\times/\G_{m}$ is Cayley.
\end{theorem}

\begin{proof}
For two $G$-varieties $X$ and $Y$ over $K$,
we write $X\sim Y$ and say that $X$ is equivalent to $Y$
if $X$ is $G$-equivariantly birationally equivalent to $Y$.

We regard $A$ also as a linear $K$-space, and we consider the projective space $\PP(A)$.
Clearly $G\sim\PP(A)$.

We denote by $t\colon A\to K$ the reduced trace.
Set
\[ V=\{a\in A\ |\ t(a)=1\}, \]
then $V$ is a $G$-variety, and it is easy to see that $\PP(A)\sim V$, hence $G\sim V$.
Since $p|n$, we have $t(x)=0$ for any $x\in K\subset A$,
hence the additive group $\G_{a}$ acts on $V$ by translations:
\[ x.a=x+a,\quad x\in K,\ a\in V\subset A.\]
Since $t(1)=0$, we can define the linear function $t$ also on $\Lie(G)=A/\langle 1\rangle$.
Set
\[ W=\{ b\in A/\langle 1\rangle\ |\ t(b)=1 \}. \]
Clearly $W=V/\G_a$.

Note that the rational map
\[ \Lie(G) =  A/\langle 1\rangle \to W\times_K \G_m, \quad b\mapsto (\,b/t(b),\, t(b)\,)\text{ for }b\in A/\langle 1\rangle \]
gives an equivalence $\Lie(G)\sim W\times_K\G_m$.
On the other hand, by Lemma \ref{lem:referee} below we have $V\sim W\times_K\G_a$.
Thus
\[
 G\sim V  \sim W\times_K\G_a\sim W\times_K\G_m   \sim \Lie(G),\]
which proves the theorem.
\end{proof}

\begin{lemma} \label{lem:referee}
$V\sim W\times_K\G_a$.
\end{lemma}

\begin{proof}
Consider the projection
\[ V\to V/\G_a=W.\]
It is enough to show that $q$ has an equivariant section.
Denote by $c_2(a)$ the second coefficient of the reduced characteristic polynomial of an element $a\in A$.
In characteristic 0, $c_2$ is of course quadratic.
But here, with our assumptions on $p$ and $n$, we have
\[c_2(a + x) = \frac{n(n-1)}{2}x^2+(n-1)t(a)x+c_2(a)=-t(a)x+c_2(a)\]
for $a\in A$ and $x\in K$ (check it in the split case for diagonal matrices, this
is easy and implies the general formula).
Hence the map $s\colon V/\G_a\to V$
sending a class $y=a+\G_a\in V/\G_a $ to the element $s(y):=a+c_2(a)\in y\subset  V$
is well defined and is an equivariant section of $q$, as required.
\end{proof}

\begin{proposition}\label{prop:referee}
If $p=2$, then $G=\PGL_{2,K}$ is not Cayley.
\end{proposition}

\begin{proof}
Indeed, assume for the sake of contradiction that there exists a $G$-equivariant birational isomorphism
\[ \varphi\colon V:=\{a\in M_2(K) \ | t(a)=1\}\birat M_2(K)/\langle 1\rangle. \]
Pick a generic invertible matrix $b\in V$.
Then $\varphi(b)$ commutes with $b$,
hence $\varphi$ restricts to a $\Z/2\Z$-equivariant birational isomorphism
\[ \psi\colon \{a\in L\ |\ t(a)=1\}\birat L/\langle 1 \rangle, \]
where $L$ is the maximal \'etale subalgebra of $M_2(K)$ generated by $b$,
and $\Z/2\Z$ is the Weyl group of $G$ with respect to its maximal torus $L^\times/\G_m$.
We split the \'etale algebra $L$ by a field extension $K'/K$ (quadratic or trivial), then  the Weyl group $\Z/2\Z$ acts on
$L\otimes_K K'=K'\times K'$ by transposition of the factors,
and we obtain a $\Z/2\Z$-equivariant birational isomorphism
\[ \psi\colon \{(a_1,a_2)\in (K')^2\ |\ a_1+a_2=1\}\birat (K')^2/\langle(1,1)\rangle.\]
But this is absurd: the Weyl group $\Z/2\Z$ acts faithfully on the left, but trivially on the right.
\end{proof}

\begin{remark}[of the referee]
\label{rem:referee}
 If  $p=2$, $2|n$,  $n\equiv 2\mod 4$, and $n\ge 6$,  then $\PGL_1(A)$ is Cayley.
(If $n\equiv 2\mod 8$, then one may use $c_4$ instead of $c_2$ in the proof;
otherwise one may cook something `linear' out of $c_4$ and powers of $c_2$.)
\end{remark}

\end{document}